\documentclass[11pt]{article}
\usepackage{amssymb, amsmath, amsfonts, tikz,  ytableau}
\usepackage{graphicx}
\usepackage{subfigure}
\usepackage{amsthm}
\usepackage[indent,margin=1cm]{caption}                 
\usepackage{float}
\usepackage[a4paper, margin=1in]{geometry}                   
\usepackage[colorlinks,linkcolor=blue,anchorcolor=blue,citecolor=blue]{hyperref}
\usepackage{diagbox}
\usepackage{pgfplots}
\usepackage{mathrsfs}
\usetikzlibrary{arrows}

\usetikzlibrary{arrows.meta}
\usetikzlibrary{positioning, calc, decorations.pathreplacing}

\numberwithin{equation}{section}
\numberwithin{figure}{section}

\hypersetup{colorlinks = true}
\newtheorem{thm}{Theorem}[section]
\newtheorem{conj}[thm]{Conjecture}

\newtheorem{cor}[thm]{Corollary}

\newtheorem{exam}[thm]{Example}
\newtheorem{prop}[thm]{Proposition}
\newtheorem{prob}[thm]{Problem}
\newtheorem{rem}[thm]{Remark}

\allowdisplaybreaks

\newcommand{\ktableau}[6]{
	\begin{tikzpicture}[scale=0.5]
		\foreach \x/\y in {#1}	{	\fill[gray!20] (\x,\y) rectangle +(1,1);}
		\foreach \x [count=\i] in {#2} {
			\foreach \y [count=\j] in \x {
				\draw[line width=0.5pt] (\j,-\i) rectangle (\j+1,-\i+1);
				\node at (\j+0.85,-\i+0.15) {\tiny{\y}};
			}
		}
		\foreach \x [count=\i] in {#3} {
			\foreach \y [count=\j] in \x {
				\node at (\j+0.5,-\i+0.5) {\small{\y}};
			}
		}
		\foreach \x/\y in {#4}	{\filldraw (\x,\y) circle (0pt) node[below right] {#5};}
		\foreach \x/\y in {#6}	{\draw[line width=1.5pt] (\x,\y) rectangle (\x+1,\y+1);}
	\end{tikzpicture}
}

\begin{document}
	\begin{center}
		{\bf \Large Newton polytopes of dual $k$-Schur polynomials}
	\end{center}

    \begin{center}
	{\bf Bo Wang$^1$, Candice X.T. Zhang$^2$ and Zhong-Xue Zhang$^3$      \\[6pt]}
	
	{\it Center for Combinatorics, LPMC\\
		Nankai University, Tianjin 300071, P. R. China\\[8pt]
		
		Email: $^1${\tt bowang@nankai.edu.cn},\ \ $^2${ \tt zhang\_xutong@mail.nankai.edu.cn},\ \ $^3${\tt zhzhx@mail.nankai.edu.cn}}
    \end{center}
	\noindent\textbf{Abstract.} 
	Rado's theorem about permutahedra and dominance order on partitions reveals that each Schur polynomial is M-convex, 
	or equivalently, it has a saturated Newton polytope and this polytope is a generalized permutahedron as well. In this paper we show that the support of each dual $k$-Schur polynomial indexed by a $k$-bounded partition coincides with that of the Schur polynomial indexed by the same partition, and hence the two polynomials share the same saturated Newton polytope. 
	The main result is based on our recursive algorithm to generate a semistandard $k$-tableau for a given shape and $k$-weight. As consequences, we obtain the M-convexity of dual $k$-Schur polynomials, affine Stanley symmetric polynomials and cylindric skew Schur polynomials. 
	
	\noindent \emph{AMS Mathematics Subject Classification 2020: }05E05, 52B05
	
	\noindent \emph{Keywords:} Saturated Newton polytope, M-convex, semistandard $k$-tableau, dual $k$-Schur polynomials, affine Stanley symmetric polynomials 
	
	\section{Introduction}
	
	Given a polynomial $f=\sum_{\alpha\in\mathbb{N}^{n}}c_{\alpha}x^{\alpha}\in\mathbb{R}[x_1,\,x_2,\,\ldots,\,x_n]$ with real coefficients, define the support of $f$, denoted ${\rm supp}(f)$, by ${\rm supp}(f)=\{\alpha\in\mathbb{N}^{n}\mid c_{\alpha}\neq0\}$. The {Newton polytope} of $f$, denoted ${\rm Newton}(f)$, is the convex hull of its exponent vectors, namely,
	\[{\rm Newton}(f)={\rm conv}\left(\alpha \mid \alpha\in{\rm supp}(f)\right)\subseteq \mathbb{R}^n.\]
	Newton polytopes have been extensively studied in various areas of mathematics since they provide a visual tool to analyze the structure of polynomials and their associated algebraic varieties. For nice expositions of Newton polytopes, see \cite{Sturm1996,GK2000,CLS2011,KKE2021}.

	Recently, the saturation of Newton polytopes has received considerable attention. 
	Following Monical, Tokcan and Yong \cite{MTY2017}, we say that a polynomial $f$ has {saturated Newton polytope}, or simply say $f$ is SNP, if 
	\[{\rm supp}(f)={\rm Newton}(f)\cap \mathbb{Z}^n .\]		
	Monical, Tokcan and Yong \cite{MTY2017} showed that various polynomials in algebraic combinatorics have saturated Newton polytopes, including Schur polynomials, Stanley symmetric polynomials, Hall-Littlewood polynomials and so on.
	
	Monical, Tokcan and Yong also proposed several conjectures on the SNP property for other polynomials, and some progress on these conjectures has been made since then. 
	Through the dual character of the flagged Weyl module, Fink, M\'esz\'aros, and {St. Dizier} \cite{FMD2018} proved the conjectured SNP property for key polynomials and Schubert polynomials. The conjecture on the SNP property for double Schubert polynomials was completely proved by Castillo, Cid-Ruiz, Mohammadi and Monta\~no \cite{CCMM2023}. 
	Monical, Tokcan, and Yong's conjecture on the SNP property for Grothendieck polynomials
	was proved by Escobar and Yong \cite{EY2017} for Grassmannian permutations, by  M\'esz\'aros and St. Dizier \cite{MSD2020} for permutations of the form $w=1w'$ with $w'$ being dominant on $\{2,3,\ldots,n\}$, and by Castillo, Cid-Ruiz, Mohammadi, and Monta\~no \cite{CCMM2022} for permutations with a zero-one Schubert polynomial. 
	The SNP property for Kronecker products of Schur polynomials was proved by Panova and Zhao \cite{PZ2023} for partitions of length two and three, and the general case is open. 
	Monical, Tokcan, and Yong's conjectures on the SNP property of Demazure atoms and Lascoux atoms remains widely open.  
	
	Motivated by Monical, Tokcan and Yong's work, the SNP property for some polynomials not mentioned in \cite{MTY2017} has also been studied. Based on the SNP property for Schur polynomials, Nguyen, Ngoc, Tuan, and Do Le Hai \cite{NNTD2023} obtained the SNP property for dual Grothendieck polynomials. Fei \cite{FJR2023} proved the SNP property for the $F$-polynomial of any rigid representation. Matherne, Morales, and Selover \cite{MMS2023} proved that the chromatic symmetric polynomials of incomparability graphs of (3+1)-free posets are SNP, though
    there does exist some chromatic symmetric polynomial which is not SNP (see \cite[Example 2.33]{MTY2017}). 

    M-convexity is another interesting property closely related to SNP. Recall that a subset $J\subset \mathbb{N}^n$ is said to be M-convex, if for all $\alpha, \beta \in J$ and any index $i$ satisfying $\alpha_{i}>\beta_{i}$, there is an index $j$ such that $\alpha_{j} <\beta_{j}$ and $\alpha-e_{i}+e_{j}\in J$, where $e_{i}$ is the $i$-th unit vector for any $i$. An immediate consequence of this definition is that an M-convex set must lie on a hyperplane. We say that a polynomial $f$ is M-convex if ${\rm supp}(f)$ is M-convex. Thus an M-convex
	polynomial must be homogeneous.	M-convexity is very essential in discrete convex analysis, which builds a connection between convex analysis and combinatorial mathematics. We refer to \cite{Mur03} for a comprehensive treatment of M-convexity. {It is known that a homogeneous polynomial $f$ is M-convex if and only if $f$ is SNP and ${\rm Newton}(f)$ is a generalized permutahedron \cite[Remark 4.1.1]{StD2020}. In fact, many of the aforementioned polynomials are M-convex, such as 
	chromatic symmetric polynomials of incomparability graphs of (3+1)-free posets \cite{MMS2023}, Schur polynomials, key polynomials, and Schubert polynomials \cite{HMMD2022}. Other progress includes Hafner, M\'esz\'aros, Setiabrata, and St. Dizier's work \cite{HMSD2023} on the M-convexity of homogenized Grothendieck polynomials of vexillary permutations.
	
	We would like to point out that the M-convexity of Schur polynomials plays a very important role in the study of the M-convexity of many other polynomials, such as Stanley symmetric polynomials and Reutenauer's symmetric polynomials.  
	As remarked by Huh, Matherne, M\'esz\'aros and St. Dizier in \cite{HMMD2022}, the M-convexity of any Schur polynomial can be deduced from its SNP property, along with the observation that its Newton polytope is a $\lambda$-permutahedron. 
	Recall that a {usual permutahedron} in $\mathbb{R}^n$ is defined by
	\[\mathcal{P}={\rm conv}(\text{permutations of }(a_1,\ldots,a_{n})\in\mathbb{R}^n) ,\]
	where the coordinates $a_1>\cdots>a_n$. By a $\lambda$-permutahedron for a partition $\lambda=(\lambda_1,\lambda_2,\ldots,\lambda_n)$, denoted $\mathcal{P}_{\lambda}$, we mean the convex hull of $S_n$-orbit of $\lambda$.
	It is known that a $\lambda$-permutahedron is a generalized permutahedron. For more information on {generalized permutahedra} see \cite{POS2009}. As pointed out by Monical, Tokcan and Yong \cite{MTY2017}, the fact that any Schur polynomial is SNP and its Newton polytope is a $\lambda$-permutahedron can be deduced from the following result due to Rado \cite{Rado}. 
	\begin{thm}[\cite{Rado}]\label{thm-Rado-Schur}
		Let $d$ be a positive integer, and $\lambda, \, \mu $ be two partitions of $d$. Then $\mathcal{P}_{\mu}\subseteq \mathcal{P}_{\lambda}$ if and only if $\mu \trianglelefteq \lambda$ (meaning that $\mu$ is less than or equal to $\lambda$ in dominance order).
	\end{thm}
	
	Along this line of investigation, the present paper is devoted to the study of Newton polytopes and M-convexity of dual $k$-Schur polynomials.
	In this paper, a function is called a polynomial if we restrict it to finite variables, and we omit the	variables for convenience.	
	The dual $k$-Schur functions appear as the dual of $k$-Schur functions with respect to the ordinary scalar product of the symmetric function space $\Lambda$ which requires that Schur functions form an orthonormal basis. It is known that both of them are generalizations of Schur functions. The $k$-Schur functions originate from the study of Macdonald positivity  conjecture \cite{LLM2003}, and play a role in the space $\mathbb{Q}[h_1,\ldots,h_k]$ analogous to the role of Schur functions in $\Lambda$, where $h_i$ denotes the $i$-th complete symmetric function. 
	Lapointe and Morse \cite{LM2008} demonstrated that dual $k$-Schur functions form a basis for  $\Lambda/\left\langle m_{\lambda}:\lambda_{1}>k\right\rangle$, where $m_{\lambda}$ are the monomial symmetric functions. The main result of this paper is that each dual $k$-Schur polynomial
	indexed by a $k$-bounded partition has the same support with the Schur polynomial indexed by the same partition. This implies that each dual $k$-Schur polynomial is M-convex and its Newton polytope is a $\lambda$-permutahedron.

    We further study the Newton polytopes and M-convexity of affine Stanley symmetric polynomials and cylindric skew Schur polynomials, the latter being special cases of the former. Affine Stanley symmetric functions were defined by Lam \cite{LAM2006}, and he also showed that dual $k$-schur functions are actually  affine Stanley symmetric functions indexed by affine Grassmannian permutations, in the same way as that Schur functions correspond to Stanley symmetric functions indexed by Grassmannian permutations. The dual $k$-Schur positivity of affine Stanley symmetric functions was first conjectured by Lam \cite{LAM2006} and then proved in his subsequent work \cite{LAM2008}. Based on Lam's work,  we obtain the M-convexity of affine Stanley symmetric polynomials.
	
	The paper is organized as follows. In Section \ref{sec-preliminaries} we will review some notations and facts on dual $k$-Schur functions. Section \ref{sec-SNP-affine-schur} is devoted to the study of the Newton polytopes and M-convexity of dual $k$-Schur polynomials. 
	In Section \ref{sec-affine-Stanley}, we will prove the M-convexity of affine Stanley symmetric polynomials and cylindric skew Schur polynomials. In Section \ref{sec-future}, we present several problems and conjectures for further research. 
	

\section{Preliminaries}\label{sec-preliminaries}
In this section, we will recall some fundamental results concerning $k+1$-cores, $k$-bounded partitions, and dual $k$-Schur functions, which will be utilized in subsequent sections. For more information, see \cite{LLMSSZ2014} and \cite{LM2005}. Let $k$ and $d$ be positive integers throughout this work without explicit mention. 

\subsection{$k+1$-cores and $k$-bounded partitions}

Both $k+1$-cores and $k$-bounded partitions are special integer partitions. By a partition $\lambda$ of $d$ we mean a sequence $\lambda=(\lambda_{1},\lambda_2, \ldots)$ 
of weakly decreasing non-negative integers satisfying $d=\lambda_{1}+\lambda_2+\cdots$. 
The length of $\lambda$, denoted $\ell(\lambda)$, is defined to be the number of its positive parts. We usually use  
$(\lambda_{1},\ldots,\lambda_{\ell})$ to represent $\lambda$
if its length is $\ell$. 
 Each partition $\lambda$ can be 
 identified with its Young diagram, which consists of boxes arranged in left-justified, with $\lambda_{i}$ boxes in the $i$-th row from bottom to top (following French notation). 
Given two partitions $\lambda$ and $\mu$ with $\mu \subseteq \lambda$ (i.e., $\mu_i\leq \lambda_i$ for all $i$), we define a skew partition $\lambda/\mu$ with its diagram consisting of boxes in $\lambda$ but not in $\mu$. 

A box $(i,j)$ in the $i$-th row and $j$-th column of the diagram is referred to as a cell. 
The hook length of a cell $(i,j)$ in $\lambda$ is defined as the number of cells directly to the right of and above $(i,j)$, counting $(i,j)$ itself once. 
A partition is a $k+1$-core if it does not contain any cell with hook length of $k+1$. The $k+1$-residue of a cell $(i,j)$ is defined as $j-i\mod (k+1)$. 
For any two cells $a=(i_{a},j_{a})$ and $b=(i_{b},j_{b})$ such that $b$ is located to the southeast of $a$, we use $h_{\lambda}(a,b)$ to denote the number of cells $(x,j_{a})$ and $(i_{b},y)$ in $\lambda$, where $i_{b}\leq x\leq i_{a}$ and $j_{a}+1\leq y\leq j_{b}$.
We say $(i,j)\in\lambda$ is a top cell if the cell $(i+1,j)\notin \lambda$.  A removable corner is a cell $(i,j)\in\lambda$ with $(i,j+1),(i+1,j)\notin \lambda$, and an addable corner is a cell $(i,j)\notin \lambda$ with $(i,j-1),(i-1,j)\in \lambda$. 
Note that any removable corner is a top cell.
For instance, in Figure \ref{fig-example-4421},  the cells $c_{1}$, $c_{2}$, $c_{3}$ and $c_{4}$ are top cells, the cells $c_{1}$, $c_{2}$, and $c_{4}$ are also  removable corners and the cell $s$ is an addable corner; the hook length of the cell $a$ is 6 and $h_{\lambda}(a,b)=5$.
\begin{figure}[htbp]
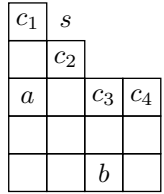

	\vspace{-0.3em}
	\centering
	\ktableau{}{{~},{~,~},{~,~,~,~},{~,~,~,~},{~,~,~,~}}{{$c_{1}$,$s$},{,$c_{2}$},{$a$,,$c_{3}$,$c_{4}$},{},{,,$b$}}{}{}{}
	\vspace{-0.5em}
	\caption{The Young diagram of  $\lambda=(4,4,4,2,1)$.}\label{fig-example-4421}
	\vspace{-1em}
\end{figure}


The following excerpt from \cite[Section 5]{LM2005} presents a fundamental result concerning cells with the same $k+1$-residues in a $k+1$-core.

\begin{prop}[\cite{LM2005}]\label{prop-same-residue}
	Let $c$ and $c^{\prime}$ be two top cells of a $k+1$-core $\gamma$, where $c$ is located weakly southeast to $c^{\prime}$. Then $c$ and $c^{\prime}$ share the same $k+1$-residue if and only if $h_{\gamma}(c^{\prime},c)$ is a multiple of $k+2$. Moreover, if $c$ and  $c^{\prime}$ have the same $k+1$-residue and $h_{\gamma}(c^{\prime},c)>k+2$, then there exists a top cell $c^{\prime\prime}$ of $\lambda$ located to the northwest of $c$ such that $h_{\gamma}(c^{\prime\prime},c)=k+2$.
\end{prop}

Let $\mathcal{C}^{k+1}$ be the set of all $k+1$-cores. 
Lapointe and Morse \cite{LM2005} established a bijection between $\mathcal{C}^{k+1}$ and a specific class of partitions known as $k$-bounded partitions. Recall that a partiton $\lambda$ is called a $k$-bounded partition if $\lambda_{1}\leq k$. Denote the set of $k$-bounded partitions of $d$ by ${\rm Par}^{k}(d)$, and set ${\rm Par}^{k}=\bigcup_{d\ge 0}{\rm Par}^{k}(d)$ with ${\rm Par}^{k}(0)$ consisting of the empty partition $\emptyset$. 
Lapointe and Morse \cite{LM2005} defined a map $\mathfrak{p}$ from $\mathcal{C}^{k+1}$ to ${\rm Par}^{k}$ by letting 
$$\mathfrak{p}(\gamma)=(\lambda_{1},\lambda_{2},\ldots,\lambda_{\ell}),$$ where $\lambda_{i}$ is the number of cells in the $i$-th row of $\gamma$ with hook lengths not exceeding $k$. 
They showed that $\mathfrak{p}$ is invertible and its inverse map $\mathfrak{c}$ can be constructed as follows: start from the top row $\lambda_{\ell}$ of the $k$-bounded partition $\lambda$ and successively move down a row; for each running row $\lambda_{i}$ if there exists a cell with hook length greater than $k$, then slide this row to the right until we reach the first position where this row has no hook lengths greater than $k$; 
continue this process until all rows have been adjusted, and we finally obtain {a skew diagram of} shape $\gamma/\rho$ (by requiring that if $\gamma/\rho$ and $\mu/\nu$ represent the same diagram then $\gamma\subseteq \mu$); let $\mathfrak{c}(\lambda)=\gamma$. Note that if the leftmost cell of $\lambda_i$ is shifted to the $j$-th column, then the top cell in column $j-1$ is a removable corner of $\mathfrak{c}(\lambda)$. For an illustration of the inverse map $\mathfrak{c}$, see the following example.  
\begin{exam}
	Let $\lambda=(4,4,4,2,1)\in {\rm Par}^{5}(15)$ and $\gamma=\mathfrak{c}(\lambda)$. Then we have the following figure.
		\begin{figure}[htbp]
	\centering
	\ytableausetup{centertableaux, boxsize=1.05em}
	\begin{minipage}{0.225\textwidth}
		\centering
		\subfigure{
			\ydiagram{0+1,0+2,0+4,0+4,0+4}
		}
		\vspace{-1em}
		\caption*{$\lambda=\mathfrak{p}(\gamma)$}
	\end{minipage}
	\begin{minipage}{0.08\textwidth}
		\centering
		\subfigure{
			$\longleftrightarrow$
		}
	\end{minipage}
	\begin{minipage}{0.25\textwidth}
		\centering
		\subfigure{
			\ydiagram{0+1,0+2,1+4,2+4,5+4}
		}
		\vspace{-1em}
		\caption*{$\gamma/\rho$}
	\end{minipage}
	\begin{minipage}{0.1\textwidth}
		\centering
		\subfigure{
			$\longleftrightarrow$
		}
	\end{minipage}
	\begin{minipage}{0.25\textwidth}
		\centering
		\subfigure{
			\ydiagram{0+1,0+2,0+5,0+6,0+9}	
		}
		\vspace{-1em}
		\caption*{$\gamma=\mathfrak{c}(\lambda)$}
	\end{minipage}
	\vspace{-0.5em}
	\caption{The inverse map $\mathfrak{c}$.}\label{fig-example-2.2}
	\vspace{-1.5em}
\end{figure}	
\end{exam}


The hook lengths of the cells of $\mathfrak{c}(\lambda)$ have the following property. 

\begin{prop}[{\cite[Lemma 4]{LM2005}}]\label{prop-gamma/rho}
	Let $\gamma/ \rho$ be the skew diagram obtained in the construction of $\mathfrak{c}$. Then 
	\vspace{-0.5em}
\begin{itemize}
	\item[(1)] the hook lengths of the cells of $\gamma/\rho$ are less than or equal to $k$;\vspace{-1em}
	\item[(2)] the boxes below $\gamma/\rho$ have hook-lengths exceeding $k+2$ in $\gamma$.
\end{itemize}	
\end{prop}

For any $k+1$-core $\gamma$, if there exists a removable corner of $\gamma$ with $k+1$-residue $i$, 
let $s_{i}(\gamma)$ denote the partition obtained by removing all removable corners of $\gamma$ with $k+1$-residue $i$, and if there exists no removable corner with $k+1$-residue $i$, then let $s_{i}(\gamma)=\gamma$. We also need the following result due to Lapointe and Morse in \cite{LM2005}.
	
\begin{prop}[{\cite[Proposition 22]{LM2005}}]\label{prop-realtion-partition-core}
	Let $\lambda\in {\rm Par}^{k}$ and $\gamma=\mathfrak{c}(\lambda)\in \mathcal{C}^{k+1}$. If there exists a removable conner of $\gamma$ with $k+1$-residue $i$, then $s_{i}(\gamma)=\mathfrak{c}(\lambda-e_{r})$, where $r$ is the highest row of $\gamma$ containing a removable corner of $k+1$-residue $i$ and $\lambda-e_{r}$ denotes the partition obtained from $\lambda$ by decreasing the $r$-th component by one. 
\end{prop}

\subsection{Dual $k$-Schur functions}

Let us first introduce the definition of Schur functions. Schur functions can be defined in many different ways \cite{StaEC2}, and here we use the expansion of monomial symmetric functions to define a Schur function. 
Given a partition $\lambda=(\lambda_1,\ldots,\lambda_\ell)$, we may identify it with the infinite sequence
$\lambda=(\lambda_1,\ldots,\lambda_\ell,0,0,\ldots)$ and define
the monomial symmetric function $m_{\lambda}$ as 
\[m_{\lambda}=\sum_{\alpha}x^{\alpha},\]
where the sum ranges over all distinct permutations $\alpha$ of $(\lambda_1,\ldots,\lambda_\ell,0,0,\ldots)$. A semistandard Young tableau (SSYT for short) of shape $\lambda$ is a filling of the Young diagram of $\lambda$ with positive integers that are weakly increasing from left to right along each row and strictly increasing from bottom to top along each column. The weight of an SSYT $T$ is the composition $\beta=(\beta_{1},\beta_{2},\ldots)$, where $\beta_{i}$ is the number of $i$'s in $T$. The number of semistandard Young tableaux of shape $\lambda$ and weight $\mu$ is denoted by $K_{\lambda,\mu}$ and is referred to as the Kostka number. 
The Schur function indexed by partition $\lambda$, denoted ${s}_{\lambda}$, is defined by
\[{s}_{\lambda}=\sum_{\mu}K_{\lambda,\mu}m_{\mu}.\]

It is interesting that the vanishing of the term $m_{\mu}$ in the expansion of ${s}_{\lambda}$ can be depicted
by the dominance order on partitions. 
For two partitions $\lambda$ and $\mu$ of $d$, we say that $\mu$ is less than or equal to $\lambda$ in dominance order, denoted $\mu\trianglelefteq \lambda$, if
\[\mu_{1}+\cdots+\mu_{i}\leq\lambda_{1}+\cdots+\lambda_{i}\quad \text{for all } i\geq 1.\]
The following result related to Kostka numbers is well-known.

\begin{prop}[\cite{StaEC2}]\label{thm-Kostka-dominance}
	Let $\lambda,\,\mu $ be two partitions of $d$. Then the Kostka number $K_{\lambda,\,\mu}\neq 0$ if and only if $\mu\trianglelefteq \lambda$.
\end{prop}
  
We proceed to introduce the definition of dual $k$-Schur functions. These functions are also known as 
affine Schur functions.
It was Lapointe and Morse who named dual $k$-Schur functions in \cite{LM2008}
by providing a definition in terms of semistandard $k$-tableau.
Dalal and Morse gave an alternative definition by using affine Bruhat counter-tableaux in \cite{DM2012}.
Dual $k$-Schur functions also appear as affine Schur functions, a class of affine Stanley symmetric functions indexed by affine Grassmannian permutations, which are introduced by Lam \cite{LAM2006}.  
For the equivalence between several equivalent formulations, see \cite{DM2012} and \cite{LLMSSZ2014}.
In this paper, we will adopt the definition given by Lapointe and Morse \cite{LM2008}.

Given $\lambda\in {\rm Par}^{k}(d)$, let $\alpha=(\alpha_{1},\alpha_{2},\ldots)$ be a composition of $d$. 
 A semistandard $k$-tableau ($k$-SSYT for short) of shape $\mathfrak{c}(\lambda)$ and $k$-weight $\alpha$ is an SSYT of shape $\mathfrak{c}(\lambda)$ such that the collection of cells filled with letter $i$ have exactly $\alpha_{i}$ distinct $k+1$-residues. For example, for $\lambda=(3,2,1,1)$ and $k=3$ there are two $k$-SSYTs of shape $\mathfrak{c}(\lambda)$ and $k$-weight $(1,2,1,2,1)$, as shown in Figure \ref{fig-ex-SSYT-kSSYT}, where the integer in the lower right corner of each cell indicates its $k+1$-residue.
 \begin{figure}[H]
	\setlength{\abovecaptionskip}{-2pt}
	\centering
	\begin{minipage}{0.25\textwidth}
		\centering
		\subfigure{
			\ktableau{}{{1},{2},{3,0,1},{0,1,2,3,0,1}}{{5},{4},{3,4,5},{1,2,2,3,4,5}}{}{}{}
		}
	\end{minipage}
	\begin{minipage}{0.25\textwidth}
		\centering
		\subfigure{
			\label{fig-ex-kSSYT-2}\ktableau{}{{1},{2},{3,0,1},{0,1,2,3,0,1}}{{5},{3},{2,4,5},{1,2,3,4,4,5}}{}{}{}
		}
	\end{minipage}
	\caption{Two $3$-SSYTs.}\label{fig-ex-SSYT-kSSYT}
	\vspace{-1.3em}
\end{figure}
The $k$-Kostka number, denoted by $K_{\lambda,\alpha}^{(k)}$, is the number of semistandard $k$-tableaux of shape $\mathfrak{c}(\lambda)$ and $k$-weight $\alpha$.
Using a generalization of the Bender-Knuth involution, Lapointe and Morse \cite{LM2007} obtained the following result.
\begin{prop}[{\cite[Corollary 25]{LM2007}}]\label{property-k-Kostka-number-2}
	Let $\lambda\in {\rm Par}^{k}(d)$ and let $\alpha$ be any composition of $d$. Then we have $K_{\lambda,\alpha}^{(k)}=K_{\lambda,p(\alpha)}^{(k)},$	where $p(\alpha)$ is the partition obtained by rearranging $\alpha$ in nonincreasing order.
\end{prop}

The dual $k$-Schur function indexed by a $k$-bounded partition $\lambda$, denoted $\mathfrak{S}_{\lambda}^{(k)}$, is defined as
\[\mathfrak{S}_{\lambda}^{(k)}=\sum_{T}x^{k\text{-weight}(T)},\]
where $T$ ranges over all $k$-SSYTs of shape $\mathfrak{c}(\lambda)$. 
By virtue of Proposition \ref{property-k-Kostka-number-2}, $\mathfrak{S}_{\lambda}^{(k)}$ is a symmetric function,
and it admits the following expansion in terms of monomial symmetric functions.  

\begin{prop}[{\cite[Proposition 6.1]{LM2008}}]\label{pop-dual-schur-to-monmial}
	For any $\lambda\in {\rm Par}^{k}$, we have
	\begin{equation}\label{eq-affine-schur-to-monamial}
		\mathfrak{S}_{\lambda}^{(k)}=m_{\lambda}+\sum_{\mu\triangleleft\lambda}K_{\lambda,\mu}^{(k)}m_{\mu}.
	\end{equation}
\end{prop}

The $k$-Kostka numbers also satisfy a triangularity property similar to the Kostka numbers, due to the following result. 

\begin{prop}[{\cite[Theorem 65]{LM2005}}]\label{property-k-Kostka-number-1}
	Let $\lambda,\,\mu\in {\rm Par}^{k}(d)$. Then
	$K_{\lambda,\mu}^{(k)}\neq 0 \text{ only if }\mu\trianglelefteq \lambda \text{, and }K_{\lambda,\lambda}^{(k)}=1.$
\end{prop}

Comparing Proposition \ref{thm-Kostka-dominance} and Proposition \ref{property-k-Kostka-number-1}, we are motivated to study whether the condition $\mu\trianglelefteq \lambda$ is also sufficient for $K_{\lambda,\mu}^{(k)}\neq 0$.
It seems that this sufficiency is still unknown, which will be explored in the next section.

\section{M-convexity of dual $k$-Schur polynomials}\label{sec-SNP-affine-schur}

The objective of this section is to establish the M-convexity of dual $k$-Schur polynomials. 
The main idea is to show that for a $k$-bounded partition $\lambda$ the dual $k$-Schur polynomial 
$\mathfrak{S}_{\lambda}^{(k)}$ has the same support with the corresponding Schur polynomial $s_{\lambda}$. 
In view of Proposition \ref{thm-Kostka-dominance} and Proposition \ref{property-k-Kostka-number-1}, 
it remains to prove the ``if'' direction of the latter proposition. 
Thus given two $k$-bounded partitions $\lambda$ and $\mu$ with $\mu\trianglelefteq \lambda $, 
to construct a $k$-SSYT of shape $\mathfrak{c}(\lambda)$ and $k$-weight $\mu$ will be our main task of this section.

It is known that if $\mu\trianglelefteq \lambda $ then there must exist an SSYT of shape $\lambda$ and weight $\mu$. However, unlike the classical case, the existence of a $k$-SSYT of shape $\mathfrak{c}(\lambda)$ and $k$-weight $\mu$ is not so evident.
The following property on the $k+1$-residues of top cells of $\mathfrak{c}(\lambda)$ is critical for generating such a $k$-SSYT.

\begin{prop}\label{lem-enough-residue} 
Suppose that $\lambda=(\lambda_1,\ldots,\lambda_{\ell})\in {\rm Par}^{k}$, and let $R(i)$ denote the set of all distinct $k+1$-residues of the top cells in the $i$-th row of $\mathfrak{c}(\lambda)$. Then, for any $1\leq i\leq \ell$, we have
\begin{align}\label{eq-topcell-difference}
	|\cup_{j=i}^{\ell}R(j)|=\lambda_{i}.
\end{align}
\end{prop}
\begin{proof}
	To prove \eqref{eq-topcell-difference}, we need to analyze the relation between the $k+1$-residues of the top cells in row $i$ and those of the top cells above row $i$. To this end, let us first introduce some notations.
	Assume that we slide the $i$-th row of $\lambda$ to right by $t_i$ boxes in $\mathfrak{c}(\lambda)$. 
	Keep in mind that we must have $t_{\ell}=0$ and $t_i\ge t_{i+1}$ for each $1\le i\le \ell-1$ according to the construction of $\mathfrak{c}(\lambda)$.
	Fixing an integer $1\le i\le \ell$, let us divide the $i$-th row of $\mathfrak{c}(\lambda)$ into four parts:
	\begin{itemize}
		\item $A(i)=\{(i,j) \mid 1\leq j\leq t_{i+1}\}$;
		\item $B(i)=\{(i,j) \mid t_{i+1}+1\leq j\leq t_{i}\}$;
		\item $C(i)=\{(i,j) \mid t_{i}+1\leq j\leq t_{i+1}+\lambda_{i+1}\}$; 
		\item $D(i)=\{(i,j) \mid t_{i+1}+\lambda_{i+1}+1\leq j\leq t_{i}+\lambda_{i}\}$;
	\end{itemize}
	where we set $t_{\ell+1}=0$. 
	It is clear that the cells of $C(i)\cup D(i)$ correspond to those in the $i$-th row of the original diagram $\lambda$, and $D(i)$ consists of all top cells in row $i$.
	We further let $\hat{A}(i)$ (respectively, $\hat{B}(i)$ and $\hat{C}(i)$) be the set of all top cells of  $\mathfrak{c}(\lambda)$ which lie above the cells in $A(i)$ (respectively, $B(i)$ and $C(i)$). 
	Given a set $P$ of cells in $\mathfrak{c}(\lambda)$, let ${\rm Res}(P)$ denote the set of distinct $k+1$-residues of the cells in $P$.
	
\begin{figure}[htbp]
	\centering	
	\begin{tikzpicture}[line cap=round,line join=round,>=triangle 45,x=0.5cm,y=0.5cm]
		\definecolor{hs2}{RGB}{175,238,238}
		\definecolor{hs1}{RGB}{204,204,225}
		\definecolor{hs4}{RGB}{175,238,238}
		\definecolor{hs3}{RGB}{192,192,192}
		\definecolor{hs5}{RGB}{188,212,230}
		\fill[line width=2.pt,color=hs1,fill=hs1,fill opacity=0.5] (0.,14.) -- (0.,13.) -- (2.,13.) -- (2.,14.) -- cycle;
		\fill[line width=2.pt,color=hs1,fill=hs1,fill opacity=0.5] (2.,12.) -- (2.,11.) -- (4.,11.) -- (4.,12.) -- cycle;
		\fill[line width=2.pt,color=hs1,fill=hs1,fill opacity=0.5] (4.,1.) -- (4.,0.) -- (0.,0.) -- (0.,1.) -- cycle;
		\fill[line width=2.pt,color=hs2,fill=hs2,fill opacity=0.5] (4.,9.) -- (4.,8.) -- (7.,8.) -- (7.,9.) -- cycle;
		\fill[line width=2.pt,color=hs2,fill=hs2,fill opacity=0.5] (7.,7.) -- (7.,6.) -- (10.,6.) -- (10.,7.) -- cycle;
		\fill[line width=2.pt,color=hs2,fill=hs2,fill opacity=0.5] (4.,1.) -- (4.,0.) -- (10.,0.) -- (10.,1.) -- cycle;
		\fill[line width=2.pt,color=hs3,fill=hs3,fill opacity=0.5] (15.,2.) -- (15.,1.) -- (18.,1.) -- (18.,2.) -- cycle;
		\fill[line width=2.pt,color=hs3,fill=hs3,fill opacity=0.5] (13.,3.) -- (13.,2.) -- (15.,2.) -- (15.,3.) -- cycle;
		\fill[line width=2.pt,color=hs3,fill=hs3,fill opacity=0.5] (10.,5.) -- (10.,4.) -- (13.,4.) -- (13.,5.) -- cycle;	
		\fill[line width=2.pt,color=hs3,fill=hs3,fill opacity=0.5] (10.,1.) -- (10.,0.) -- (18.,0.) -- (18.,1.) -- cycle;
		\fill[line width=2.pt,color=hs4,fill=hs4,fill opacity=0.5] (18.,1.) -- (18.,0.) -- (26.,0.) -- (26.,1.) -- cycle;
		
		\fill[line width=2.pt,color=hs3,fill=hs3,fill opacity=0.5] (2.,1.) -- (2.,2.) -- (4.,2.) -- (4.,1.) -- cycle;
		\fill[line width=2.pt,color=hs5,fill=hs5,fill opacity=0.5] (4.,1.) -- (4.,2.) -- (15.,2.) -- (15.,1.) -- cycle;

		\draw [line width=0.8pt, color=blue] (10.,1.)-- (10.,0.);
		\draw [line width=0.8pt, color=blue] (10.,1.)-- (26.,1.);
		\draw [line width=0.8pt, color=blue] (26.,0.)-- (10.,0.);
		\draw [line width=0.8pt, color=blue] (26.,1.)-- (26.,0.);
		\draw [line width=0.4pt, color=blue] (18.,1.)-- (18.,0.);
		\draw [line width=0.4pt, color=blue] (19.,1.)-- (19.,0.);
		\draw [line width=0.4pt, color=blue] (25.,1.)-- (25.,0.);
		
		\draw [line width=0.4pt,color=red] (4.,2.)-- (18.,2.);
		\draw [line width=0.4pt,color=red] (18.,2.)-- (18.,1.);
		\draw [line width=0.4pt,color=red] (4.,2.)-- (4.,1.);
		\draw [line width=0.4pt,color=red] (4.,1.)-- (10.,1.);
		\draw [line width=0.4pt] (17.,2.)-- (17.,1.);
		
		\draw [line width=0.4pt] (2.,3.)-- (2.,2.);
		\draw [line width=0.4pt] (2.,2.)-- (4.,2.);
		\draw [line width=0.4pt] (15.,3.)-- (15.,2.);
		\draw [line width=0.4pt] (13.,3.)-- (15.,3.);
		\draw [line width=0.4pt] (13.,5.)-- (13.,3.);
		\draw [line width=0.4pt] (10.,4.)-- (11.,4.);
		\draw [line width=0.4pt] (11.,4.)-- (11.,5.);
		\draw [line width=0.4pt] (10.,5.)-- (10.,4.);
		\draw [line width=0.4pt] (10.,5.)-- (13.,5.);
		\draw [line width=0.4pt] (10.,7.)-- (10.,5.);
		\draw [line width=0.4pt] (7.,7.)-- (10.,7.);
		\draw [line width=0.4pt] (4.,9.)-- (7.,9.);
		\draw [line width=0.4pt] (7.,9.)-- (7.,7.);
		\draw [line width=0.4pt] (4.,9.)-- (4.,12.);
		\draw [line width=0.4pt] (4.,12.)-- (2.,12.);
		\draw [line width=0.4pt] (2.,14.)-- (2.,12.);
		\draw [line width=0.4pt] (0.,14.)-- (2.,14.);
		\draw [line width=0.4pt] (2.,3.)-- (13.,3.);
		\draw [line width=0.4pt,dotted] (4.,12.)-- (4.,14.);
		\draw [line width=0.4pt,dotted] (10.,6.)-- (10.,9.);
		\draw [line width=0.4pt,dotted] (18.,2.)-- (18.,5.);
		\draw [line width=0.4pt,dashed] (0.,0.)-- (0.,14.);
		
		\node at (18.5,0.5){$d_{1}$};
		\node at (25.5,0.5){$d_{2}$};
		\node at (10.5,4.5){$c_{1}$};
		\node at (17.5,1.5){$c_{2}$};
		\node at (-1,0.5){\color{blue}$i$};
		\node at (-1,1.5){\color{red}$i+1$};		
		\node at (-1,2.5){$i+2$};
		\node at (-1,13.5){$\ell$};	
		
		\draw[decoration={brace,mirror}, decorate, thick] (0,-0.2)--(4,-0.2);
		\node at (2, -1)[font=\footnotesize]{$A(i)$};
		\draw[decoration={brace,mirror}, decorate, thick] (4,-0.2)--(10,-0.2);
		\node at (7,-1)[font=\footnotesize]{$B(i)$};
		\draw[decoration={brace,mirror}, decorate, thick] (10,-0.2)--(18,-0.2);
		\node at (14, -1)[font=\footnotesize]{$C(i)$};
		\draw[decoration={brace,mirror}, decorate, thick] (18,-0.2)--(26,-0.2);
		\node at (22, -1)[font=\footnotesize]{$D(i)$};

		\draw[decoration={brace}, decorate, thick] (0,14.1)--(4,14.1);
		\node at (2, 14.6)[font=\footnotesize]{$\hat{A}(i)$};
		\draw[decoration={brace}, decorate, thick] (4,9.1)--(10,9.1);
		\node at (7, 9.6)[font=\footnotesize]{$\hat{B}(i)$};
		\draw[decoration={brace}, decorate, thick] (10,5.1)--(18,5.1);
		\node at (14.5, 5.6)[font=\footnotesize]{$\hat{C}(i)$};	
		
		\draw[decoration={brace,mirror}, decorate, thick] (0,0.8)--(2,0.8);
		\node at (1, 0.4)[font=\tiny]{$A(i+1)$};
		\draw[decoration={brace,mirror}, decorate, thick] (2,0.8)--(4,0.8);
		\node at (3, 0.4)[font=\tiny]{$B(i+1)$};
		\draw[decoration={brace,mirror}, decorate, thick] (4,0.8)--(15,0.8);
		\node at (9.5, 0.4)[font=\tiny]{$C(i+1)$};
		\draw[decoration={brace,mirror}, decorate, thick] (15,.8)--(18,.8);
		\node at (16.5, .4)[font=\tiny]{$D(i+1)$};
	\end{tikzpicture}
	\vspace{-0.6em}
	\caption{{The sets $A,\,B,\,C,\, D$ and cells $b_i,\,c_i,\,d_i$ for $i=1,\,2$ of shape $\mathfrak{c}(\lambda)$.} }\label{fig-pr-lem--enough-residue}
	\vspace{-1em}
\end{figure}
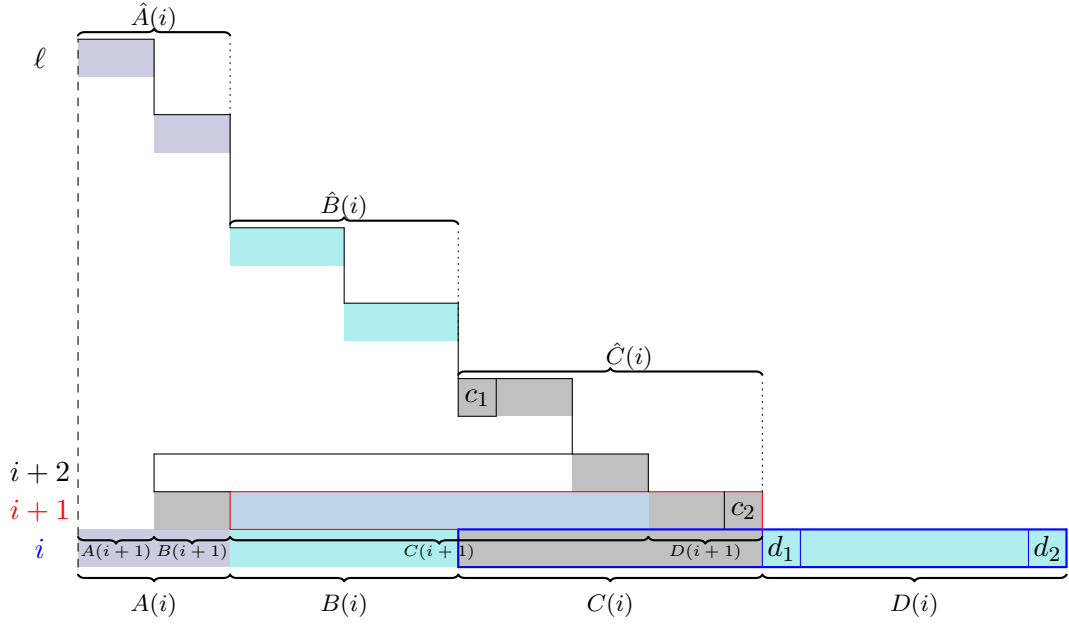
	
	Since $|D(i)|\leq \lambda_i<k+1$, all of its cells have distinct $k+1$-residues, and hence 
	$|D(i)|=|{\rm Res}(D(i))|$. 
	By the construction of $\mathfrak{c}(\lambda)$, for any two cells $c,c'\in\hat{C}(i)$, we have $h_{\mathfrak{c}(\lambda)}(c,c')<k+1$. From Proposition \ref{prop-same-residue} it follows that all cells of $\hat{C}(i)$ have distinct $k+1$-residues, and hence $|\hat{C}(i)|=|{\rm Res}(\hat{C}(i))|$. 
	Thus, 
	\begin{align}\label{eq-lambda}
		\lambda_{i}=|{C}(i)|+|D(i)| =|\hat{C}(i)|+|D(i)|=|{\rm Res}(\hat{C}(i))|+|{\rm Res}(D(i))|.
	\end{align}
	Note that 
	\begin{align}\label{eq-residue}
		|\cup_{j=i}^{\ell}R(j)|= |{\rm Res}(\hat{A}(i))\cup {\rm Res}(\hat{B}(i)) \cup {\rm Res}(\hat{C}(i))\cup {\rm Res}(D(i))|.
	\end{align}
	By \eqref{eq-lambda} and \eqref{eq-residue}, we see that \eqref{eq-topcell-difference} is implied by the following claim.
	
	Claim: For any $1\leq i\leq {\ell}$, we have
	\begin{itemize}
		\item[(I)] ${\rm Res}(\hat{A}(i))\subseteq {\rm Res}(\hat{B}(i))\cup {\rm Res}(\hat{C}(i))$;
		\item[(II)] ${\rm Res}(\hat{C}(i))\cap {\rm Res}(D(i))=\emptyset$;
		\item[(III)] ${\rm Res}(\hat{B}(i))\subseteq {\rm Res}(D(i))$.
	\end{itemize}
	
	To prove (I) we use induction on $\ell-i$. 
	Since $\lambda_{\ell}\leq \lambda_1\leq k$, we have $t_{\ell}=0$, which implies that $|D(\ell)|=\lambda_{\ell}$ and $\hat{A}({\ell})=\hat{B}({\ell})=\hat{C}({\ell})=\emptyset$. Thus, (I) holds for $i=\ell$. 
	Now, assume that (I) holds for $i+1$ and consider the case of $i$.  Let us first note that
	\begin{align}
		\hat{A}(i)&\supseteq \hat{A}(i+1),\label{eq-aa}\\
		\hat{A}(i)&= \hat{A}(i+1)\uplus \hat{B}(i+1), \label{eq-aab}\\ 		
		\hat{B}(i)\uplus \hat{C}(i) &= \hat{C}(i+1)\uplus {D}(i+1),\label{eq-bccd}
	\end{align}
	and these relations are evident from the definitions of $\hat{A}(i),\hat{B}(i),\hat{C}(i)$ and $D(i)$, as indicated in Figure \ref{fig-pr-lem--enough-residue}.
	By \eqref{eq-aab}, \eqref{eq-bccd} and the induction hypothesis, we have
	\begin{align*}
		\hat{A}(i)=\hat{A}(i+1)\uplus \hat{B}(i+1) \subseteq \hat{C}(i+1)\uplus D(i+1) = \hat{B}(i)\cup \hat{C}(i),
	\end{align*}
	i.e., (I) holds for $i$, and the proof of (I) is complete.
	
	We proceed to prove (II) and (III).
	For (II) we may assume that neither $\hat{C}(i)$  nor $D(i)$ is empty. 
	For (III) we may assume that $\hat{B}(i)$ is not empty. In fact, we can also assume that 
	$D(i)$ is not empty since if $D(i)=\emptyset$ then  $\hat{B}(i)=\emptyset$ by their definitions. 
	From now on, we may assume that neither of $\hat{B}(i),\,\hat{C}(i)$ and $D(i)$ is empty.
	
	To prove (II) and (III), we introduce a few more notations. We label the leftmost cell in $D(i)$ as $d_1$, and the rightmost cell in $D(i)$ as $d_2$. Similarly, we label the leftmost cell in $\hat{C}(i)$ as $c_1$, and the rightmost cell in $\hat{C}(i)$ as $c_2$. For the relative positions of $c_1,c_2,d_1$ and $d_2$, see Figure \ref{fig-pr-lem--enough-residue}.
	
	Now we can prove (II). In view of the fact that $h_{\mathfrak{c}(\lambda)}(c_1,d_2)\leq k<k+2$, for any $p,q\in \hat{C}(i)\cup D(i)$ we have $h_{\mathfrak{c}(\lambda)}(p,q)<k+2$.
	By Proposition \ref{prop-same-residue}, this means that ${\rm Res}(\hat{C}(i))\cap {\rm Res}(D(i))=\emptyset$, as desired.
	
	Finally, we prove (III). By Proposition \ref{prop-gamma/rho}, it is clear that for any $p\in \hat{B}(i)$ we have $h_{\mathfrak{c}(\lambda)}(p,d_2)\geq k+2$ and  $h_{\mathfrak{c}(\lambda)}(p,d_1)=h_{\mathfrak{c}(\lambda)}(p,c_2)+2\leq k+2$, where the equality follows from the relative position of $d_1$ and $c_2$. Thus, there exists some $q\in D(i-1)$ such that $h_{\mathfrak{c}(\lambda)}(p,q)=k+2$.
	By Proposition \ref{prop-same-residue} again, we find that the two cells $p$ and $q$ have the same $k+1$-residue. This implies that ${\rm Res}(\hat{B}(i))\subseteq {\rm Res}(D(i))$. 
	
	This completes the proof of the claim and hence that of the proposition. 
\end{proof}	

Now we are almost ready to prove the existence of a $k$-SSYT of shape $\mathfrak{c}(\lambda)$ and $k$-weight $\mu$ for any pair of $k$-bounded partitions $\lambda=(\lambda_1,\ldots,\lambda_\ell)$ and $\mu=(\mu_1,\ldots,\mu_{\iota})$ satisfying $\mu \trianglelefteq \lambda$. The fundamental principle to generate such a $k$-SSYT is to place the largest numbers as top and right as possible.
In some sense our construction is inspired by Fayers' construction of an SSYT $T$ of shape $\lambda$ and weight $\mu$ \cite{FAY2019}, which we recall below. 
We use the symbol $j$ to denote the maximal row index that satisfies $\lambda_j\ge \mu_{\iota}$.
To construct $T$, first fill the top cells of the first $\mu_{\iota}$ columns with $\iota$'s in the Young diagram of $\lambda$, and then slide these $\iota$'s to the end of their respective rows. More precisely, for each $i>j$  the number of $\iota$'s assigned to the end of row $i$ is $\lambda_i-\lambda_{i+1}$ (set $\lambda_{\ell +1}=0$), and the number of $\iota$'s assigned to the end of row $j$ is $\mu_{\iota}-\lambda_{j+1}$. It is important to highlight that these occurrences of $\iota$'s are exactly positioned in the top cells of each row. 
Now let $\hat{\lambda}$ denote the partition obtained from $\lambda$ by removing the cells filled by $\iota$'s,
and let $\hat{\mu}=(\mu_1,\ldots,\mu_{\iota-1})$. Fayers noted that $\hat{\mu} \trianglelefteq \hat{\lambda}$.
Then fill $\mu_{\iota-1}$ top cells of $\hat{\lambda}$ with $\iota-1$'s in the same manner. Iterating the above process will eventually produce an SSYT $T$ of shape $\lambda$ and weight $\mu$. (Note that we need not to assume that $\mu$ and $\lambda$ are $k$-bounded for Fayers' construction.)

 For an illustration of Fayers' construction, see the following example. 

\begin{exam}\label{example-SSYT-44421}
	Let $\lambda=(4,4,4,2,1)$, $\mu=(3,3,3,3,3)$. The following figure describes the procedure to generate an SSYT $T$ of shape $\lambda$ and weight $\mu$ according to Fayers' construction:
	\begin{figure}[H]
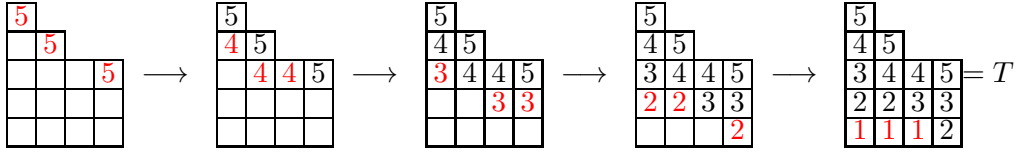

		\centering
		\ytableausetup{centertableaux, boxsize=0.95em}
		\ytableaushort{{\textcolor{red}{5}},\none {\textcolor{red}{5}},\none \none \none {\textcolor{red}{5}} ,\none, \none}* {1,2,4,4,4}
		$~\longrightarrow~$
		\ytableaushort{5,{\textcolor{red}{4}}5,\none{\textcolor{red}{4}} {\textcolor{red}{4}} 5,\none, \none}* {1,2,4,4,4}
		$~\longrightarrow~$
		\ytableaushort{5,45,{\textcolor{red}{3}}445,\none\none{\textcolor{red}{3}}{\textcolor{red}{3}},\none}* {1,2,4,4,4}
		$~\longrightarrow~$
		\ytableaushort{5,45,3445,{\textcolor{red}{2}}{\textcolor{red}{2}}33,\none\none\none{\textcolor{red}{2}}}* {1,2,4,4,4}
		$~\longrightarrow~$
		\ytableaushort{5,45,3445,2233,{\textcolor{red}{1}}{\textcolor{red}{1}}{\textcolor{red}{1}}2}* {1,2,4,4,4}
		$=T$
		\caption{ The construction of the SSYT $T$.}
		\vspace{-1.5em}
	\end{figure}
\end{exam}

To generate a $k$-SSYT of shape $\mathfrak{c}(\lambda)$ and $k$-weight $\mu$ for a pair of $k$-bounded partitions $\lambda=(\lambda_1,\ldots,\lambda_\ell)$ and $\mu=(\mu_1,\ldots,\mu_{\iota})$ satisfying $\mu \trianglelefteq \lambda$, the basic operation of our algorithm consists of the row filling of $\iota$'s into the Young diagram $\mathfrak{c}(\lambda)$, denoted by $\mathfrak{c}(\lambda)\leftarrow \{\ell(\mu)\}$. (Note that $\ell(\mu)=\iota$.)
The operation $\mathfrak{c}(\lambda)\leftarrow \{\ell(\mu)\}$ consists of the following steps: 
\begin{itemize}
	\item[\bf{R1:}] Find the maximal row index $j$ such that $\lambda_j\ge \mu_{\iota}$. (Such a $j$ always exists since $\mu \trianglelefteq \lambda$.) 
	
	\item[\bf{R2:}] If $j=\ell$, simply mark the rightmost $\mu_{\iota}$ top cells of the $\ell$-th row of $\mathfrak{c}(\lambda)$ with bold boundaries, place $\iota$'s in these marked cells, and then go to Step {\bf{R3}}. (It is clear that all these top cells occupied by $\iota$ have different $k+1$-residues.)
	On the other hand, if $j<\ell$, perform the filling process from row $\ell$ to row $j$ in the following way (from top to bottom):
	 
	\begin{itemize}
		\item[{\bf{R2-1:}}] Begin with row $\ell$, mark all the top cells of the $\ell$-th row of $\mathfrak{c}(\lambda)$  with bold boundaries, and then place $\iota$'s in these marked cells.
		
		\item[{\bf{R2-2:}}] If $i>j$ and all rows above row $i$ have been  performed the filling operation, mark the rightmost $\lambda_{i}-\lambda_{i+1}$ top cells of the $i$-th row of $\mathfrak{c}(\lambda)$  with bold boundaries such that these cells have different $k+1$-residues with those cells already filled with $\iota$'s. (This is possible by Proposition \ref{lem-enough-residue}.) Let $c_i$ denote the leftmost cell of these marked cells. Then place $\iota$'s in $c_i$ and all the cells to the right of $c_i$ in the $i$-th row.

		\item[{\bf{R2-3:}}] For row $j$ mark the rightmost $\mu_{\iota}-\lambda_{j+1}$ top cells of the $j$-th row of $\mathfrak{c}(\lambda)$  with bold boundaries such that these cells have different $k+1$-residues with those cells already filled with $\iota$'s. (Again by Proposition \ref{lem-enough-residue} this is possible.) Similarly, let $c_j$ denote the leftmost cell of these $\mu_{\iota}-\lambda_{j+1}$ top cells. Then place $\iota$'s in $c_j$ and all the cells to the right of $c_j$ in the $j$-th row.  
	\end{itemize}

	\item[{\bf{R3:}}] Find the rightmost removable corner $c$ filled with $\iota$ among all rows above row $j-1$ in $\mathfrak{c}(\lambda)$. If $c$ has $k+1$-residue $y$, place $\iota$'s in all removable corners of $\mathfrak{c}(\lambda)$ with $k+1$-residue $y$. Shade all removable corners with $k+1$-residue $y$.  
	
	\item[{\bf{R4:}}] Suppose that $r$ is the highest row of $\mathfrak{c}(\lambda)$ containing a removable corner with $k+1$-residue $y$ and let $\lambda-e_r$ be the partition defined as in Proposition \ref{prop-realtion-partition-core}. (Deleting all shaded removable corners from $\mathfrak{c}(\lambda)$ will lead to a diagram of $\mathfrak{c}(\lambda-e_r)$.)	
	If there exists some marked cell above row $j-1$ in $\mathfrak{c}(\lambda-e_r)$, replace $\lambda$ with $\lambda-e_r$ 
	and then go to Step {\bf{R3}}.  Otherwise, the filling process stops.
\end{itemize}

\begin{rem}\label{rem-hr}
All cells filled with $\iota$'s in $\mathfrak{c}(\lambda)$ form a horizontal strip. Moreover, these cells have exactly $\mu_\iota$ distinct $k+1$-residues. 
\end{rem}

We will now provide an example to illustrate the above filling process. 

\begin{exam}
	Let $\lambda=(4,4,4,2,1)$, $\mu=(3,3,3,3,3)$, and $k=5$. The filling operation   $\mathfrak{c}(\lambda)\leftarrow \{5\}$ is carried out as follows:
	\begin{figure}[htbp]
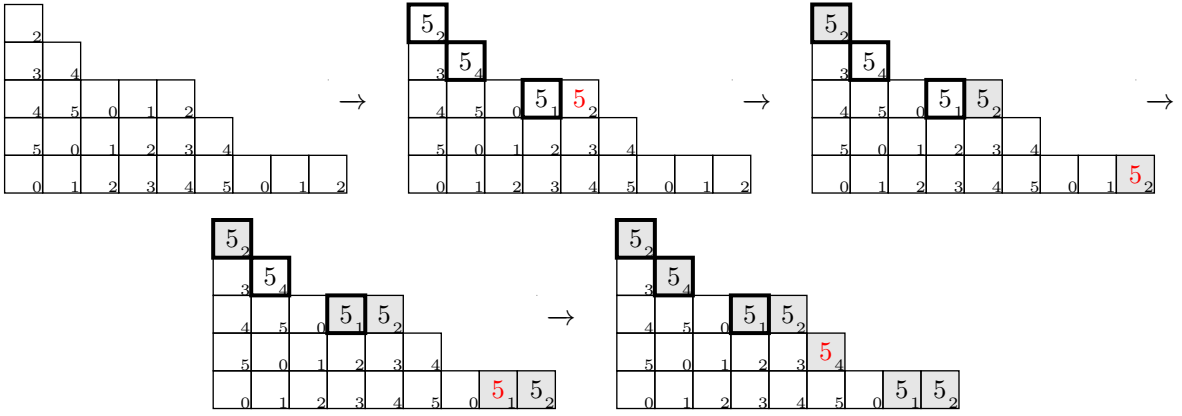

		\centering
		\ktableau{}{{2},{3,4},{4,5,0,1,2},{5,0,1,2,3,4},{0,1,2,3,4,5,0,1,2}}{}{9.5/-2}{$\xrightarrow{~}$}{}
		\ktableau{}{{2},{3,4},{4,5,0,1,2},{5,0,1,2,3,4},{0,1,2,3,4,5,0,1,2}}{{5},{~,5},{~,~,~,5,\textcolor{red}{5}},{~,~,~,~,~,~},{~,~,~,~,~,~,~,~,~}}{9.5/-2}{$\xrightarrow{~}$}{1/-1,2/-2,4/-3}
		\ktableau{1/-1,5/-3,9/-5}{{2},{3,4},{4,5,0,1,2},{5,0,1,2,3,4},{0,1,2,3,4,5,0,1,2}}{{5},{~,5},{~,~,~,5,5},{~,~,~,~,~,~},{~,~,~,~,~,~,~,~,\textcolor{red}{5}}}{9.5/-2}{$\xrightarrow{~}$}{1/-1,2/-2,4/-3}\\[5pt]
		\ktableau{1/-1,5/-3,9/-5,4/-3,8/-5}{{2},{3,4},{4,5,0,1,2},{5,0,1,2,3,4},{0,1,2,3,4,5,0,1,2}}{{5},{~,5},{~,~,~,5,5},{~,~,~,~,~,~},{~,~,~,~,~,~,~,\textcolor{red}{5},5}}{9.5/-2}{$\xrightarrow{~}$}{1/-1,2/-2,4/-3}
		\ktableau{1/-1,5/-3,9/-5,4/-3,8/-5,2/-2,6/-4}{{2},{3,4},{4,5,0,1,2},{5,0,1,2,3,4},{0,1,2,3,4,5,0,1,2}}{{5},{~,5},{~,~,~,5,5},{~,~,~,~,~,\textcolor{red}{5}},{~,~,~,~,~,~,~,5,5}}{}{}{1/-1,2/-2,4/-3}
		\caption{ The filling operation   $\mathfrak{c}(\lambda)\leftarrow \{5\}$.}
		\vspace{-1em}
	\end{figure}
\end{exam}

There is a close relation between 
our filling operation and Fayers' construction. In fact, we have the following property, which is of great use in the construction of a $k$-SSYT of shape $\lambda$ and $k$-weight $\mu$. 

\begin{prop}\label{lem-delete-l}
Given a pair of $k$-bounded partitions $\lambda=(\lambda_1,\ldots,\lambda_{\ell})$ and $\mu=(\mu_1,\ldots,\mu_{\iota})$ with $\mu \trianglelefteq \lambda$, let $T^{\{k\}}_{\iota}$ denote the resulting partial k-SSYT of the row filling $\mathfrak{c}(\lambda)\leftarrow \{\ell(\mu)\}$, and let $T$ be the resulting SSYT of Fayers' construction. 
Then the diagram obtained from $T^{\{k\}}_{\iota}$ by removing all $\iota$'s is of shape $\mathfrak{c}(\hat{\lambda})$, where $\hat{\lambda}$ is the shape of the tableau obtained from $T$ by removing all $\iota$'s.
\end{prop}

\begin{proof}
	According to the filling rule in $\mathfrak{c}(\lambda)\leftarrow \{\ell(\mu)\}$, all cells filled with $\iota$'s in $T_{\iota}^{\{k\}}$ can be grouped into $\mu_{\iota}$ sets according to their $k+1$-residues.     
Then removing all $\iota$'s from $T_{\iota}^{\{k\}}$ can be done by successively deleting these sets in the order they are shaded in Step R4, say these ordered sets are $A_1,A_2,\ldots,A_{\mu_{\iota}}$.  

Let $j$ denote the maximal index 
such that $\lambda_j\geq \mu_{\iota}$.
	We first claim that, in Step R4 of the row filling $\mathfrak{c}(\lambda)\leftarrow \{\ell(\mu)\}$, when we find the rightmost removable corner $c$ with $k+1$-residue $y$ among all rows above row $j-1$ in $\mathfrak{c}(\lambda)$, all filled top cells with $k+1$-residue $y$ must be removable corners.  
On the one hand, the filled cells with $k+1$-residue $y$ below $c$ clearly are removable corners by the filling process. On the other hand, suppose there is a filled top cell $c^{\prime}$ with $k+1$-residue $y$ above $c$ such that there is a cell $c^{\prime\prime}$ to the right of $c^{\prime}$. According to Proposition \ref{prop-same-residue}, without loss of generality, let $h_{\mathfrak{c}(\lambda)}{(c^{\prime},c)}=k+2$. Then $h_{\mathfrak{c}(\lambda)}{(c^{\prime\prime},c)}=k+1$, which contradicts the fact that $\mathfrak{c}(\lambda)$ is a $k+1$-core. This completes the proof of the claim. 

Let $\bar{c}$ denote the topmost removable corner with $k+1$-residue $y$ in Step R4. 
If $\bar{c}$ lies in row $r$, then, by Proposition \ref{prop-realtion-partition-core} and the above claim, deleting all shaded cells with $k+1$-residue $y$ in $T_{\iota}^{\{k\}}$ yields a partial SSYT of shape $\mathfrak{c}(\lambda-e_r)$.
By Step R2, we know that $\bar{c}$ is the unique marked cell with $k+1$-residue $y$. 

Based on the above arguments, only one marked cell is removed when we delete some $A_i$ from the diagram each time.  
Note that in the row filling $\mathfrak{c}(\lambda)\leftarrow \{\ell(\mu)\}$
all marked cells have distinct $k+1$-residues. Moreover, there are $\lambda_i-\lambda_{i+1}$ marked cells in row $i$ for each $i>j$ and $\mu_{\iota}-\lambda_{j+1}$ marked cells in row $j$. 
By Fayers' construction $\hat{\lambda}$ is just the partition obtained from $\lambda$ by decreasing the $i$-th part of $\lambda$ by $\lambda_i-\lambda_{i+1}$ for each $i>j$ and decreasing the $j$-th part of $\lambda$ by $\mu_{\iota}-\lambda_{j+1}$. 
Thus, deleting all $A_i$'s from $T^{\{k\}}_{\iota}$ will eventually lead to a diagram of $\mathfrak{c}(\hat{\lambda})$. 
The proof is complete. 	
\end{proof}

The first main result of this section is as follows, which can be considered as a refinement of Proposition \ref{thm-Kostka-dominance}.

\begin{thm}\label{thm-k-Kostka-number}
	Let $\lambda=(\lambda_1,\ldots,\lambda_\ell),\,\mu=(\mu_1,\ldots,\mu_{\iota})\in {\rm Par}^k(d)$. Then $K_{\lambda,\mu}^{(k)}\neq 0$  if and only if  $\mu\trianglelefteq \lambda$.
\end{thm}
\begin{proof}
	By Proposition \ref 
	{property-k-Kostka-number-1}, it suffices to show there exists a $k$-SSYT of shape $\mathfrak{c}(\lambda)$ and $k$-weight $\mu$ if $\mu\trianglelefteq \lambda$. 
	Such a tableau can be produced based on the above row filling operation and Fayers' construction.
	 
	It is well known that each SSYT of shape $\lambda$ corresponds to a chain of partitions from $\emptyset$ to $\lambda$. 
	As before, let $T$ be the resulting SSYT of shape $\lambda=(\lambda_1,\ldots,\lambda_\ell)$ and weight $\mu=(\mu_1,\ldots,\mu_\iota)$ generated by Fayers' construction. Suppose that $T$ corresponds to the following partition sequence
	\begin{align}\label{eq-chain-of-partitions}
		\emptyset=\lambda^{(0)}\subseteq \lambda^{(1)}\subseteq\cdots\subseteq\lambda^{(\iota-1)}\subseteq\lambda^{(\iota)}=\lambda,
	\end{align}
	namely, all cells of the skew diagram $\lambda^{(i)}/\lambda^{(i-1)}$ are filled with $i$ for all $1\leq i\leq\iota$. For each $1\leq i\leq \iota$ let $\mu^{(i)}=(\mu_1,\ldots,\mu_i)$. Then from Fayers' construction it follows that if $\mu \trianglelefteq \lambda$ then $\mu^{(i)} \trianglelefteq {\lambda}^{(i)}$ for every $1\leq i\leq \iota$.
	
	We associate a $k$-SSYT $T^{\{k\}}$ of shape $\mathfrak{c}({\lambda})$ and $k$-weight $\mu$
	as follows.
	Begin with $T_{\iota+1}^{\{k\}}=\mathfrak{c}(\lambda^{(\iota)})=\mathfrak{c}(\lambda)$, the empty $k$-SSYT of shape $\mathfrak{c}(\lambda)$.  
	If $i>0$ and $T_{i+1}^{\{k\}}$ is defined, then let  
	$T_{i}^{\{k\}}$ be the partial SSYT of shape $\mathfrak{c}(\lambda)$ obtained from $T_{i+1}^{\{k\}}$ by applying the row filling $\mathfrak{c}(\lambda^{(i)})\leftarrow \{\ell(\mu^{(i)})\}$. 
	Moreover, by the proof of Proposition \ref{lem-delete-l}, the empty cells of $T_{i}^{\{k\}}$ form a diagram of shape $\mathfrak{c}(\lambda^{(i-1)})$. 
	The process ends at $T_{1}^{\{k\}}$, and let $T^{\{k\}}=T_{1}^{\{k\}}$. 
	
	We proceed to show that $T^{\{k\}}$ is a $k$-SSYT of shape $\mathfrak{c}(\lambda)$ and $k$-weight $\mu$. Its semistandness is ensured since for each $1\leq i\leq \iota$ all cells filled with $i$'s form a horizontal strip of shape $\mathfrak{c}(\lambda^{(i)})/\mathfrak{c}(\lambda^{(i-1)})$ in $\mathfrak{c}(\lambda)$, as stated in Remark \ref{rem-hr}. 
    Moreover, by Remark \ref{rem-hr}, for each $1\leq i\leq \iota$ the cells filled with $i$ in $T^{\{k\}}$ have exactly $\mu_i$ distinct $k+1$-residues, in view of that these cells are just those filled by the row filling  $\mathfrak{c}(\lambda^{(i)})\leftarrow \{\ell(\mu^{(i)})\}$. This completes the proof. 
\end{proof} 

We use $(\lambda,\mu,k)\rightarrow T^{\{k\}}$ to denote the algorithm for generating $T^{\{k\}}$ from $\lambda$ and $\mu$ in the above proof. 
For an illustration of this algorithm, see the following example.

\begin{exam}
	Let $\lambda=(4,4,4,2,1),\,\mu=(3,3,3,3,3)\in \mathrm{Par}^{5}(15)$. By the algorithm $(\lambda,\mu,5)\rightarrow T^{\{5\}}$, we obtain a $5$-SSYT of shape $\mathfrak{c}(\lambda)$ and $5$-weight $\mu$ based on the partition sequence: $\lambda^{(1)}=(3),\, \lambda^{(2)}=(4,2),\, \lambda^{(3)}=(4,4,1),\, \lambda^{(4)}=(4,4,3,1),\, \lambda^{(5)}=(4,4,4,2,1)  $.
	\begin{figure}[H]
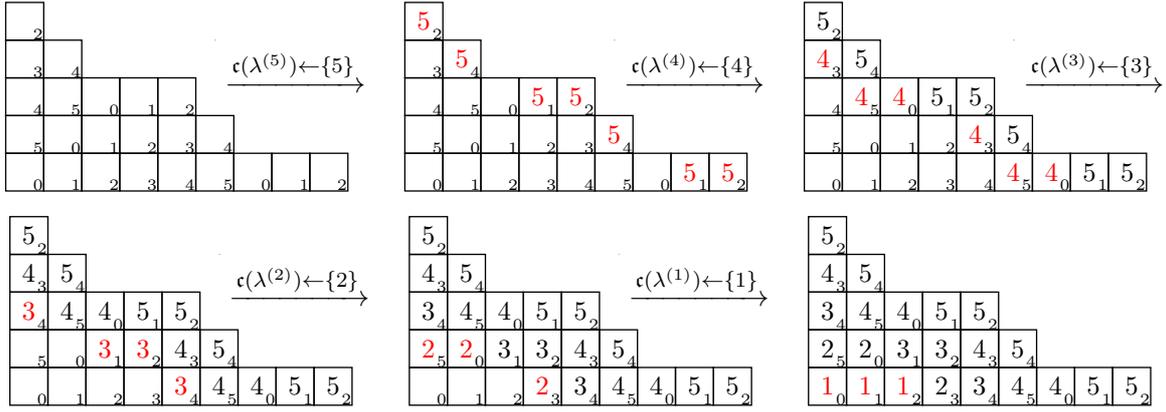

		\centering
		\ktableau{}{{2},{3,4},{4,5,0,1,2},{5,0,1,2,3,4},{0,1,2,3,4,5,0,1,2}}{}{6.5/-1}{$\xrightarrow{\mathfrak{c}(\lambda^{(5)})\leftarrow \{5\}}$}{}
		\ktableau{}{{2},{3,4},{4,5,0,1,2},{5,0,1,2,3,4},{0,1,2,3,4,5,0,1,2}}{{\textcolor{red}{5}},{~,\textcolor{red}{5}},{~,~,~,\textcolor{red}{5},\textcolor{red}{5}},{~,~,~,~,~,\textcolor{red}{5}},{~,~,~,~,~,~,~,\textcolor{red}{5},\textcolor{red}{5}}}{6.5/-1}{$\xrightarrow{\mathfrak{c}(\lambda^{(4)})\leftarrow \{4\}}$}{}
		\ktableau{}{{2},{3,4},{4,5,0,1,2},{5,0,1,2,3,4},{0,1,2,3,4,5,0,1,2}}{{5},{\textcolor{red}{4},5},{~,\textcolor{red}{4},\textcolor{red}{4},5,5},{~,~,~,~,\textcolor{red}{4},5},{~,~,~,~,~,\textcolor{red}{4},\textcolor{red}{4},5,5}}{6.5/-1}{$\xrightarrow{\mathfrak{c}(\lambda^{(3)})\leftarrow \{3\}}$}{}\\[5pt]
		\ktableau{}{{2},{3,4},{4,5,0,1,2},{5,0,1,2,3,4},{0,1,2,3,4,5,0,1,2}}{{5},{4,5},{\textcolor{red}{3},4,4,5,5},{~,~,\textcolor{red}{3},\textcolor{red}{3},4,5},{~,~,~,~,\textcolor{red}{3},4,4,5,5}}{6.5/-1}{$\xrightarrow{\mathfrak{c}(\lambda^{(2)})\leftarrow \{2\}}$}{}
		\ktableau{}{{2},{3,4},{4,5,0,1,2},{5,0,1,2,3,4},{0,1,2,3,4,5,0,1,2}}{{5},{4,5},{3,4,4,5,5},{\textcolor{red}{2},\textcolor{red}{2},3,3,4,5},{~,~,~,\textcolor{red}{2},3,4,4,5,5}}{6.5/-1}{$\xrightarrow{\mathfrak{c}(\lambda^{(1)})\leftarrow \{1\}}$}{}
		\ktableau{}{{2},{3,4},{4,5,0,1,2},{5,0,1,2,3,4},{0,1,2,3,4,5,0,1,2}}{{5},{4,5},{3,4,4,5,5},{2,2,3,3,4,5},{\textcolor{red}{1},\textcolor{red}{1},\textcolor{red}{1},2,3,4,4,5,5}}{}{}{}
		\caption{The construction of the $5$-SSYT of shape $\mathfrak{c}(\lambda)$ and $5$-weight $\mu$.}
		\vspace{-1em}	
	\end{figure} 
\end{exam}

We are now in a position to present the M-convexity of dual $k$-Schur polynomials.
\begin{thm}\label{thm-M-SNP-dual-k-Schur}
Let $\lambda\in \mathrm{Par}^{k}$. Then the dual $k$-Schur polynomial $\mathfrak{S}^{(k)}_{\lambda}$ has SNP and its Newton polytope is a $\lambda$-permutahedron $\mathcal{P}_{\lambda}$.
Equivalently, $\mathfrak{S}^{k}_{\lambda}$ has a M-convex support.
\end{thm}
\begin{proof}	
Combining Proposition \ref{thm-Kostka-dominance}, Proposition \ref{pop-dual-schur-to-monmial}, Theorem \ref{thm-k-Kostka-number} and the definition of $s_{\lambda}$ it is clear that the support of $\mathfrak{S}_{\lambda}^{(k)}$ is the same as that of $s_{\lambda}$ for any fixed $k$-bounded partition $\lambda$. 
Thus, the statements of the theorem immediately follow from the above arguments, together with Theorem \ref{thm-Rado-Schur} and \cite[Remark 4.1.1]{StD2020}, respectively.
\end{proof}

\section{M-convexity of affine Stanley symmetric polynomials}\label{sec-affine-Stanley}
In this section, we further obtain the M-convexity of affine Stanley symmetric polynomials. As a corollary, we also derive that property for the cylindric skew Schur polynomials, a special subclass of the affine Stanley symmetric polynomials.

The affine Stanley symmetric functions were introduced by Lam \cite{LAM2006} analogous to the Stanley symmetric functions \cite{STAN1984}. These functions turn out to have a natural geometric interpretation, and 
they represent Schubert classes of the cohomology of the affine Grassmannian \cite{LAM2008}. There are several ways to define the affine Stanley symmetric functions; for more information see \cite{LAM2006,LAM2008,LLMS2010,YUN2013}.
In this paper we adopt one combinatorial definition given in \cite{LAM2006}.

Let $\widetilde{S}_{k+1}$ be the  affine symmetric group with generators $s_{0},s_{1},\ldots,s_{k}$ satisfying the affine Coxeter relations:
\begin{equation*}
	\begin{array}{cl}
		s_{i}^2=id & \text{ for all } i, \\
		s_{i}s_{i+1}s_{i}=s_{i+1}s_{i}s_{i+1}& \text{ for all } i,\\
		s_{i}s_{j}=s_{j}s_{i} & \text{ for } |i-j|\neq 1\mod (k+1).
	\end{array}
\end{equation*}
Note that the symmetric group $S_{k+1}$ embeds in $\widetilde{S}_{k+1}$ as the subgroup generated by the elements $s_{1},\ldots,s_{k}$. For an affine permutation $w\in \widetilde{S}_{k+1}$, a shortest expression of $w$ is called a reduced word. The length of $w$, denoted by $\ell(w)$, is the length of its reduced word. 

We would like to remark that each generator $s_i$ can be considered as an action on a $k+1$-core
$\gamma$ by removing its removable corners with $k+1$-residue $i$, or adding all addable corners with $k+1$-residue $i$, or doing nothing if there are no removable or addable corners with $k+1$-residue $i$.

The affine symmetric group can be realized as the set of all bijections $w:\mathbb{Z}\to\mathbb{Z}$ such that $w(i+k+1)=w(i)+k+1$ for all $i$, and $\sum_{i=1}^{k+1}w(i)=\sum_{i=1}^{k+1}i$. The code of  $w$, denoted as $c(w)$, is a  sequence $(c_{1},c_{2},\ldots,c_{k+1})$ of non-negative integers, where $c_{i}$ is the number of indices $j$ such that $j>i$ and $w(j)<w(i)$. 

An affine permutation $w\in\widetilde{S}_{k+1}$ is called cyclically decreasing if there exists a reduced word $s_{i_{1}}s_{i_{2}}\cdots s_{i_{l}}$ of $w$ such that each generator is distinct, and whenever $s_{i}$ and $s_{i+1}$ both occur, $s_{i+1}$ precedes $s_{i}$. A cyclically decreasing decomposition of $w\in\widetilde{S}_{k+1}$ is an expression of $w=w^{{1}}w^{{2}}\cdots w^{{r}}$ such that each $w^{i}$ is cyclically decreasing and $\ell(w)=\sum_{i=1}^{r}\ell(w^{i})$. Now we give the definition of affine Stanley symmetric function.

Given an affine permutation $w\in\widetilde{S}_{k+1}$, the affine Stanley symmetric function $\widetilde{F}_{w}$ is defined by 
\begin{align}\label{def-affine-stanley}
	\widetilde{F}_{w}=\sum_{w=w^{{1}}w^{{2}}\cdots w^{{n}}}x_{1}^{\ell(w^{1})}x_{2}^{\ell(w^{2})}\cdots x_{r}^{\ell(w^{r})},
\end{align}
where the summation is taken over all cyclically decreasing decomposition of $w$.
When $w$ is an affine Grassmannian permutation,
$\widetilde{F}_w$ is called an affine Schur function by Lam \cite{LAM2006}, who pointed out that these affine Schur functions are just the dual $k$-Schur functions defined by Lapointe and Morse \cite{LM2008}. When $w\in S_{k+1}$,  $\widetilde{F}_w$ is the usual Stanley symmetric function \cite{LAM2006}.

It is worth mentioning that the M-convexity of  Stanley symmetric polynomials can be obtained from their Schur positivity and the following theorem. 

\begin{thm}[{\cite[Proposition 2.5]{MTY2017}}]\label{thm-SNP-Schur-combin}
	Let $f=\sum_{\mu} c_{\mu}s_{\mu}$ be a homogeneous symmetric polynomial with degree $d$. If there exists a partition $\lambda$ such that $c_{\lambda}\neq 0$ and $c_{\mu}\neq0$ only if $\mu\trianglelefteq \lambda$, then ${\rm Newton}(f)=\mathcal{P}_{\lambda}$. Moreover, if  $c_{\mu}\geq 0$ for all $\mu$, then $f$ has SNP, and hence it is M-convex. 
\end{thm}
We want to use Theorem \ref{thm-SNP-Schur-combin} to establish the M-convexity of the affine Stanley symmetric polynomials. Unfortunately, the affine Stanley symmetric polynomials are not always Schur positive, and the following gives such an example.
\begin{exam}
	Let $w=s_{2}s_{1}s_{0}s_{2}\in\widetilde{S}_{3}$. Since $w$ only has a reduced word $s_{2}s_{1}s_{0}s_{2}$, all cyclically decreasing decompositions of $w$ are
	\[(s_{2})(s_{1})(s_{0})(s_{2}),\,(s_{2}s_{1})(s_{0})(s_{2}),\,(s_{2})(s_{1}s_{0})(s_{2}),\,(s_{2})(s_{1})(s_{0}s_{2}),\,(s_{2}s_{1})(s_{0}s_{2}).\]
	Thus, we have
	\[\widetilde{F}_{w}=m_{1111}+m_{211}+m_{22}=s_{22}-s_{1111}.\]
\end{exam}

Fortunately, Lam \cite{LAM2008} showed that any affine Stanley symmetric function can be expressed as a non-negative linear combination of dual $k$-Schur functions.
 
\begin{thm}[{\cite[Corollary 8.5]{LAM2008}}]\label{thm-affine-schur-positive-of-affine-Stanley}
	For any $w\in \widetilde{S}_{k+1}$, the affine Stanley symmetric functions $\widetilde{F}_{w}$ expand positively in terms of dual $k$-Schur functions.
\end{thm}

Lam \cite{LAM2006}  also proved the existence of a dominant term in the expansion of $\widetilde{F}_{w}$ in terms of dual $k$-Schur functions.

\begin{thm}[{\cite[Theorem 21]{LAM2006}}]\label{thm-affine-schur-expansion-of-affine-Stanley}
	Suppose that $w\in \widetilde{S}_{k+1}$ and  
	$\widetilde{F}_{w}=\sum_{\lambda}a_{w\lambda}\mathfrak{S}_{\lambda}^{(k)}$. 
	If $a_{w\lambda}\neq 0$ then $\lambda\trianglelefteq\mu(w)$, where 
	$\mu(w)$ is the partition conjugate to the partition obtained by rearranging the parts of $c(w^{-1})$ in weakly decreasing order. 
	Moreover, $a_{w\mu(w)}=1$.	
\end{thm}

In order to obtain the M-convexity of affine Stanley polynomials, we now give a criterion for determining whether the linear combinations of dual $k$-Schur polynomials are SNP or not. 

\begin{thm}\label{thm-SNP-affine-Schur-combin}
	Suppose that $f$ is a homogeneous symmetric polynomial of degree $d$ in $n$ variables and it has the following expansion
	\begin{align}\label{eq-f-expansion}
		f=\sum_{\mu\trianglelefteq \lambda}c_{\mu}\mathfrak{S}^{(k)}_{\mu},
	\end{align} 
	where $\lambda$ is a $k$-bounded partition of $d$. 
If there exists a partition $\lambda$ such that $c_{\lambda}\neq0$ and $c_{\mu}\neq0$ only if $\mu\trianglelefteq \lambda$, then ${\rm Newton}(f)=\mathcal{P}_{\lambda}$. 
	Moreover, if $c_{\mu}\geq0$ for all $\mu$, then $f$ has SNP and hence it is M-convex.
\end{thm}

\begin{proof}
	Note that if $\ell(\lambda)>n$, then $c_{\lambda}=0$ and therefore we may assume that $n\geq \ell(\lambda)$.
	Let us first prove that 
	\[{\rm Newton}(f)=\mathcal{P}_{\lambda}.\] 
	Substituting  \eqref{eq-affine-schur-to-monamial} into  \eqref{eq-f-expansion}, we obtain
	\[f=\sum_{\mu\trianglelefteq\lambda}d_{\mu}m_{\mu}.\]
	By Proposition \ref{property-k-Kostka-number-1}, we see that $d_{\lambda}\neq 0$.
	For each partition $\mu$, by definition we have ${\rm Newton}(m_{\mu})=\mathcal{P}_{\mu}$.
	For any different partitions $\mu,\,\nu$ and real numbers $a,b$ with $ab\neq 0$, note that
	\[{\rm Newton}(a\cdot m_{\mu}+b\cdot m_{\nu})={\rm conv}\left({\rm Newton}(m_{\mu})\cup{\rm Newton}(m_{\nu})\right).\] 
	Thus,
	\begin{equation}\label{eq-pf-affine-schur-newton-f}
		{\rm Newton}(f)={\rm conv}\left(\bigcup_{\mu\trianglelefteq\lambda,\,d_{\mu}\neq 0}{\rm Newton}(m_{\mu})\right)={\rm conv}\left(\bigcup_{\mu\trianglelefteq\lambda,\,d_{\mu}\neq 0}\mathcal{P}_{\mu}\right).
	\end{equation}
	Since $\mathcal{P}_{\mu}\subseteq \mathcal{P}_{\lambda}$ for any $\mu\trianglelefteq\lambda$ by Theorem \ref{thm-Rado-Schur}, we immediately obtain that ${\rm Newton}(f)=\mathcal{P}_{\lambda}$ as desired.
	
	Next we prove that $f$ has SNP.
	Assume that $\alpha\in \mathcal{P}_{\lambda}$ is a lattice point, then it suffices to show $\alpha\in {\rm supp}(f)$. 
	We also denote $p(\alpha)$ as the partition obtained by rearranging $\alpha$ in weakly decreasing order. Hence, $\mathcal{P}_{p(\alpha)}=\mathcal{P}_{\alpha}\subseteq \mathcal{P}_{\lambda}$. 
	Then Theorem \ref{thm-Rado-Schur} implies that $p(\alpha)\trianglelefteq\lambda$. From Theorem \ref{thm-k-Kostka-number} it follows that $K_{\lambda,p(\alpha)}^{(k)}\neq0$. 
	Consequently, $m_{p(\alpha)}$ appears in the expansion of $\mathfrak{S}_{\lambda}^{(k)}$ by Proposition \ref{pop-dual-schur-to-monmial}. 
	Since $x^{\alpha}$ appears in $m_{p(\alpha)}$, it follows that $x^{\alpha}$ appears in $\mathfrak{S}_{\lambda}^{(k)}$. 
	Moreover, by the condition $c_{\mu}\geq 0$ and $c_{\lambda}>0$, we conclude that $x^{\alpha}$ will not vanish in $f$, i.e., $\alpha\in {\rm supp}(f)$. This completes the proof.
\end{proof}

We are now in a position to present the main result of this section.
\begin{thm}\label{thm-S-M-affine-Ssf}
	For any $w\in \widetilde{S}_{k+1}$, the affine Stanley symmetric polynomial $\widetilde{F}_{w}$ is M-convex.
\end{thm}
\begin{proof}
The statements of the theorem immediately follow from Theorem \ref{thm-affine-schur-positive-of-affine-Stanley}, Theorem \ref{thm-affine-schur-expansion-of-affine-Stanley} and Theorem \ref{thm-SNP-affine-Schur-combin}.	  
\end{proof}

The above theorem also enables us to obtain the M-convexity of cylindric skew Schur functions, which were introduced by Postnikov \cite{POS2005}.
To be self-contained, we also give an overview of related definitions following \cite{LAM2006}.

A cylindric shape $\lambda$ is an infinite lattice path in $\mathbb{Z}^2$, consisting only of steps upwards and to the left, invariant under shifts $(m-n,m)$ where $1\leq m\leq n-1$. Let $C^{n,m}$ be the set of cylindric shapes. For any two cylindric shapes $\lambda,\,\mu\in C^{n,m}$, we say that $\mu\subseteq\lambda$, if $\mu$ always lies weakly to the left of $\lambda$. If $\mu\subseteq\lambda$, then we call $\lambda/\mu$ a cylindric skew shape.

Given a cylindric skew shape $\lambda/\mu$, a semistandard cylindric skew tableau of shape $\lambda/\mu$ and weight $\alpha=(\alpha_{1},\alpha_{2},\ldots,\alpha_{\ell})$ is a chain of cylindric shapes in $C^{n,m}$, i.e, 
\[\mu=\lambda^{0}\subseteq\lambda^{1}\subseteq\cdots\subseteq\lambda^{\ell}=\lambda,\]
such that the cylindric skew shape $\lambda^{i}/\lambda^{i-1}\, (1\leq i\leq\ell)$ contains at most one box in each column and $\alpha_{i}$ boxes in any $n-m$ consecutive columns.
\begin{exam}
	Let $n=5$, $m=2$ and $\alpha=(1,2,1,2)$. Putting  $i$ into the boxes of $\lambda^{i}/\lambda^{i-1}$, we obtain a semistandard cylindric skew tableau of weight $(1,2,1,2)$, as shown in Figure \ref{fig-SSCST}.   
	\begin{figure}[H]
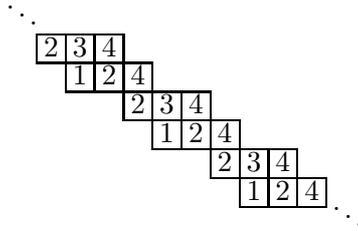

		\centering
		\ytableausetup{centertableaux,boxsize=0.95em}
		\begin{ytableau}
			\none[\ddots] \\
			\none&2&3&4\\
			\none&\none&1&2&4\\
			\none&\none&\none&\none&2&3&4\\
			\none&\none&\none&\none&\none&1&2&4\\
			\none&\none&\none&\none&\none&\none&\none&2&3&4\\
			\none&\none&\none&\none&\none&\none&\none&\none&1&2&4\\
			\none&\none&\none&\none&\none&\none&\none&\none&\none&\none&\none&\none[~\ddots]\end{ytableau}
		\caption{A semistandard cylindric skew tableau with weight $\alpha$.}
		\label{fig-SSCST}
		\vspace{-1em}
	\end{figure}
\end{exam}

For a cylindric skew shape $\lambda/\mu$, the cylindric skew Schur function $s_{\lambda/\mu}^{c}$ is defined as 
\[s_{\lambda/\mu}^{c}=\sum_{T}x^{\text{weight}(T)},\]
summing over all semistandard cylindric skew tableaux of shape $\lambda/\mu$.

The cylindric skew Schur functions generalizes usual Schur functions.  McNamara proved in \cite{MCN2006} that, with the exception of trivial cases, the cylindric skew Schur functions are not Schur positive in general. 

Lam \cite{LAM2006} demonstrated that the cylindric skew Schur functions are indeed special cases of the affine Stanley symmetric functions indexed by 321-avoiding affine permutations. This conclusion was later reproved by Lee \cite{LEE2019}.

\begin{thm}[{\cite[Corollary 5]{LEE2019}}]\label{thm-cylindric-skew-schur}
	For a cylindric shape $\lambda/\mu\in C^{n,m}$, there exists a 321-avoiding affine permutation $w\in \widetilde{S}_{n}$ such that the cylindric skew Schur function $s_{\lambda/\mu}^{c}$ is the same as $\widetilde{F}_{w}$.
\end{thm}

Combining Theorem \ref{thm-S-M-affine-Ssf} and \ref{thm-cylindric-skew-schur}, we immediately obtain the M-convexity of the cylindric skew Schur polynomials.

\begin{cor}\label{thm-SNP-cylindric-skew-Schur-function}
The cylindric skew Schur polynomial $s_{\lambda/\mu}^{c}$ is M-convex for any cylindric skew shape $\lambda/\mu$.
\end{cor}

\section{Future directions}\label{sec-future}

In this section we present some open problems and conjectures for further research. 

As shown in Section \ref{sec-affine-Stanley}, we obtain the M-convexity of affine Stanley symmetric polynomials based on the M-convexity of dual $k$-Schur polynomials. One of the key ingredients of our approach is Theorem \ref{thm-affine-schur-positive-of-affine-Stanley},
a deep result obtained by Lam \cite{LAM2008}.  
It would be interesting to find a direct proof of Theorem \ref{thm-S-M-affine-Ssf} based on the definition given by \eqref{def-affine-stanley}.  

\begin{prob}\label{prob-affine-stanley}
	Find a combinatorial proof of the M-convexity of affine Stanley symmetric polynomials.
\end{prob}

In this paper we obtain the M-convexity of cylindric skew Schur polynomials as a direct consequence of 
the M-convexity of affine Stanley symmetric polynomials. Thus, we may ask a question for cylindric skew Schur polynomials similar to Problem \ref{prob-affine-stanley}.
 
\begin{prob}
	Find a combinatorial proof of the M-convexity of cylindric skew Schur polynomials based on their tableau interpretation. 
\end{prob}

Once the M-convexity of dual $k$-Schur polynomials is established, it is natural to ask whether $k$-Schur polynomials are M-convex. 
Form  Proposition \ref{property-k-Kostka-number-1}, we know that the inverse of the matrix $\|K^{(k)}\|_{\lambda,\mu\in{{\rm Par}^{k}}}$ exists. The $k$-Schur functions, indexed by $k$-bounded partitions, are defined by inverting the unitriangular system:
\[h_{\lambda}=s_{\lambda}^{(k)}+\sum_{\mu\rhd\lambda}K^{(k)}_{\mu,\lambda}s_{\mu}^{(k)}\, \text{ for all }\, 
\lambda_{1}\leq k.\]
Here we use the definition in \cite{LM2005}.
It is known that $s_{\lambda}^{(k)}$ is always Schur positive \cite{BMPD2019}. 
For any $1\leq d\leq 25$, $1\leq k\leq 25$, and any partition $\lambda\in \mathrm{Par}^k(d)$,
we find that there always exists a dominant term in the Schur expansion of $s_{\lambda}^{(k)}$ by using SageMath \cite{SAGE}, which implies the M-convexity of $s_{\lambda}^{(k)}$ by Theorem \ref{thm-SNP-Schur-combin}. We have the following conjecture. 

\begin{conj}
	For any $\lambda\in \mathrm{Par}^k(d)$ the $k$-Schur polynomials $s_{\lambda}^{(k)}$ are M-convex.
\end{conj}

We can also study the M-convexity of homogeneous polynomials by using the theory of Lorentzian polynomials, developed by Br\"and\'en and Huh \cite{BH20},
who showed that the support of any Lorentzian polynomial is ${\rm M}$-convex.
Given a polynomial $f=\sum_{\alpha\in\mathbb{N}^n}c_{\alpha}x^{\alpha}$, define its normalization by
\[N(f)=\sum_{\alpha\in\mathbb{N}^n}c_{\alpha}\frac{x_1^{\alpha_1}}{\alpha_1!}\cdots\frac{x_{n}^{\alpha_{n}}}{\alpha_{n}!}.\]
It is known that the normalization of Schur polynomials $s_{\lambda}$ for any partition $\lambda$ is a Lorentzian polynomial \cite{HMMD2022}. As a generalization of Schur polynomials, it is natural to ask whether the normalized $k$-Schur polynomials and dual $k$-Schur polynomials are Lorentzian polynomials. We propose the following conjectures.

\begin{conj}\label{conj-1-lp}
	For any $\lambda\in \mathrm{Par}^k(d)$ the polynomial $N(s_{\lambda}^{(k)})$ is a Lorentzian polynomial.
\end{conj}

\begin{conj}\label{conj-2-lp}
	For any $\lambda\in \mathrm{Par}^k(d)$ the polynomial $N(\mathfrak{S}_{\lambda}^{(k)})$ is a Lorentzian polynomial.
\end{conj}

By using SageMath \cite{SAGE}, we verify Conjecture \ref{conj-1-lp} for $1\leq k\leq 9$ and all $k$-bounded partitions of size less than or equal to 9, and we verify Conjecture \ref{conj-2-lp} for $1\le k \le 9$ and all $k+1$-cores of size less than or equal to 9. All functions are restricted to 9 variables.

	\vskip 0.5cm
	\noindent \textbf{Acknowledgments.} This work is supported by the Fundamental Research Funds for the Central Universities and the National Science Foundation of China (11971249). We would also like to thank Xin-Bei Liu for the helpful discussions.

\end{document}